\documentclass[12pt,a4paper]{article}

\usepackage{fullpage}
\usepackage{amsthm,amsmath,amssymb,amscd,amsfonts}
\usepackage{graphicx}
\usepackage[colorlinks=true,citecolor=black,linkcolor=black,urlcolor=blue]{hyperref}
\usepackage[utf8]{inputenc}
\usepackage[T1]{fontenc}

\theoremstyle{plain}
\newtheorem{theorem}{Theorem}[section]
\newtheorem{lemma}[theorem]{Lemma}

\newtheorem{proposition}[theorem]{Proposition}

\newtheorem{claim}[theorem]{Claim}

\theoremstyle{definition}

\theoremstyle{remark}
\newtheorem*{remark}{Remark}

\newcommand{\N}{\mathbb{N}}
\newcommand{\Z}{\mathbb{Z}}

\newcommand{\Zd}{\Z/2\Z}
\newcommand{\Zt}{\Z/3\Z}
\newcommand{\Zm}{\Z/m\Z}
\newcommand{\m}{\mathfrak{m}}
\newcommand{\Sc}{\mathrm{S}}
\newcommand{\WS}{\mathrm{WS}}
\newcommand{\T}{\mathrm{T}}

\title{\bf Modular Schur numbers}

\author{Jonathan Chappelon\thanks{Part of this work was carried out during a short postdoctoral research visit at the Instituto de Matemáticas de la Universidad de Sevilla Antonio de Castro Brzezicki (IMUS).}\\
\small Institut de Mathématiques et de Modélisation de Montpellier\\[-0.8ex]
\small Université Montpellier 2, France\\
\small\tt jonathan.chappelon@um2.fr\\
\and
María Pastora Revuelta Marchena \qquad  María Isabel Sanz Domínguez\\
\small Departamento de Matemática Aplicada I\\[-0.8ex]
\small Escuela Técnica Superior de Ingeniería de Edificación\\[-0.8ex]
\small Universidad de Sevilla, Spain\\
\small\tt \{pastora,isanz\}@us.es
}

\date{June 17, 2013}

\begin{document}

\maketitle

\begin{abstract}
  For any positive integers $l$ and $m$, a set of integers is said to be (weakly) $l$-sum-free modulo $m$ if it contains no (pairwise distinct) elements $x_1,x_2,\ldots,x_l,y$ satisfying the congruence $x_1+\ldots+x_l\equiv y\bmod{m}$. It is proved that, for any positive integers $k$ and $l$, there exists a largest integer $n$ for which the set of the first $n$ positive integers $\{1,2,\ldots,n\}$ admits a partition into $k$ (weakly) $l$-sum-free sets modulo $m$. This number is called the generalized (weak) Schur number modulo $m$, associated with $k$ and $l$. In this paper, for all positive integers $k$ and $l$, the exact value of these modular Schur numbers are determined for $m=1$, $2$ and $3$.

\end{abstract}

\section{Introduction}

In \cite{Guy}, Guy proposed two unsolved problems in elementary number theory. The first one is the Schur's problem of partitioning integers into sum-free classes (Problem E12). Schur proved in \cite{Schur} that if the set of the first $k!e$ positive integers is partitioned into $k$ parts any way, then $x+y=z$ can be solved in integers within one part. Let $\Sc(k)$ be the largest integer $n$ such that there exists a partition of the first $n$ positive integers $\{1,2,\ldots,n\}$ into $k$ parts with no solution to the equation
$$
x+y=z,
$$
in any part. The exact value of $\Sc(k)$ is known only for $k\in\{1,2,3,4\}$ and the Problem E12 is to determine it for $k\geq5$. The second unsolved problem proposed by Guy is a modular version of this Schur's problem (Problem E13). It was posed by Abbott and Wang in \cite{Abbott}. Let $\T(k)$ be the largest integer $n$ such that there exists a partition of $\{1,2,\ldots,n\}$ into $k$ parts, with no solution to the congruence
$$
x+y\equiv z\pmod{n+1},
$$
in any part. Abbott and Wang determined that $\T(k)=\Sc(k)$ for
$k\in\{1,2,3,4\}$, and they conjectured that the equality is true
for all positive integers $k$. The main purpose of this paper is to study and explicitly determine modular generalizations of Schur numbers.

\subsection{Modular generalized Schur numbers}

Let $k$ be a positive integer and let $S$ be a set of integers. A $k$-partition of $S$ is a set $P=\{S_1,\ldots,S_k\}$ of subsets of $S$ such that any element of $S$ is contained in exactly one element of $P$. Let $l$ be a positive integer. A set of integers is said to be \textit{$l$-sum-free} if it contains no elements $x_1,\ldots,x_l,y$ satisfying
$$
x_1+\ldots+x_l=y.
$$
For every positive integer $k$, the \textit{generalized Schur number} $\Sc(k,l)$ is the largest integer $n$ for which the set of the first $n$ positive integers $\{1,2,\ldots,n\}$ admits a $k$-partition into $l$-sum-free sets.
\par For $l=2$, the numbers $\Sc(k,2)=\Sc(k)$ are known as \textit{Schur numbers}. They have been introduced by Schur himself in 1916 \cite{Schur}, in order to study a modular version of the Fermat's Last Theorem. He proved that those numbers are always finite, for every positive integer $k$. The first few Schur numbers are given in Table~\ref{tab2}.
\begin{table}[!ht]
\begin{center}
\begin{tabular}{|c|c|c|c|c|c|c|c|}
\hline
$k$ & $1$ & $2$ & $3$ & $4$ & $5$ & $6$ & $7$ \\
\hline
$\Sc(k,2)$ & $1$ & $4$ & $13$ & $44$ & $160\leq \cdots \leq 305$ & $\geq 536$ & $\geq 1680$ \\
\hline
\end{tabular}
\end{center}
\caption{\label{tab2}The first few Schur numbers $\Sc(k,2)$.}
\end{table}
\par The exact value of $\Sc(4,2)$ was obtained by Baumert \cite{Abbott2}. The lower and upper bounds of $\Sc(5,2)$ are due to Exoo \cite{Exoo} and Sanz \cite{Sanz}, respectively. Finally, the lower bounds of $\Sc(6,2)$ and $\Sc(7,2)$ were obtained by Fredricksen and Sweet \cite{Fredricksen} by considering symmetric sum-free partitions.
\par Many generalizations of Schur numbers have appeared since their introduction. In this paper, the generalized Schur numbers that we consider are similarly defined in \cite{Beutelspacher,Landman}. These numbers are always finite (see \cite{Rado,Landman} for instance).
\par Let $m$ be a positive integer. A set of integers is said to be \textit{$l$-sum-free modulo $m$} if it contains no elements $x_1,\ldots,x_l,y$ satisfying
$$
x_1+\ldots+x_l\equiv y\pmod{m}.
$$
For every positive integer $k$, the \textit{generalized Schur number modulo $m$}, denoted by $\Sc_m(k,l)$, is the largest integer $n$ for which the set of the first $n$ positive integers $\{1,2,\ldots,n\}$ admits a $k$-partition into $l$-sum-free sets modulo $m$.
\par Obviously, for every modulus $m$, the inequality
\begin{equation}\label{eq1}
\Sc_m(k,l)\leq \Sc(k,l)
\end{equation}
holds because a $l$-sum-free set modulo $m$ of integers is also $l$-sum-free. Moreover, since $m+\ldots+m\equiv m\bmod{m}$, a $l$-sum-free set of integers modulo $m$ does not contain the integer $m$. Therefore, we have
\begin{equation}
\Sc_m(k,l)\leq m-1.
\end{equation}
\par For $l=2$, Abbott and Wang investigated in \cite{Abbott} the numbers
$$
\T(k)=\max\left\{n\in\N\ \middle|\ \Sc_{n+1}(k,2)=n\right\},
$$
where $k$ is a positive integer. They obtained that $\T(k)=\Sc(k,2)$ for $k\in\{1,2,3,4\}$ and they conjectured that the equality is true for all positive integers $k$.
\par In this paper, we explicitly determine the modular generalized Schur numbers $\Sc_m(k,l)$ for small values of $m$: for all moduli $m\in\{1,2,3\}$. For $m=1$, the result is clear. Indeed,
$$
\Sc_1(k,l)=0,
$$
for all $k\geq1$ and $l\geq1$, since every positive integer $x$ verifies $x+\ldots+x \equiv x\bmod{1}$ and thus, there does not exist non-empty $l$-sum-free set modulo $1$. For $m=2$ and $m=3$, the exact values of $\Sc_m(k,l)$ are given by the following theorems.
\begin{theorem}\label{thm1}
Let $k$ and $l$ be two positive integers. Then,
$$
\Sc_2(k,l) = \left\{\begin{array}{lll}
0 & \text{for} & l\ \text{odd},\\
1 & \text{for} & l\ \text{even}.
\end{array}
\right.
$$
\end{theorem}
\begin{theorem}\label{thm2}
Let $k$ and $l$ be two positive integers. Then,
$$
\Sc_3(k,l) = \left\{\begin{array}{lll}
0 & \text{for} & k\geq1\ \text{and}\ l\equiv 1\bmod{3},\\
1 & \text{for} & k=1\ \text{and}\ l\equiv 0,2\bmod{3},\\
2 & \text{for} & k\geq2\ \text{and}\ l\equiv 0,2\bmod{3}.
\end{array}\right.
$$
\end{theorem}
A simple proof of these theorems will be provided in Section~\ref{secsm}.

\subsection{Modular generalized weak Schur numbers}

A set of integers is said to be \textit{weakly $l$-sum-free} if it contains no pairwise distinct elements $x_1,\ldots,x_l,y$ satisfying
$$
x_1+\ldots+x_l=y.
$$
For every positive integer $k$, the \textit{generalized weak Schur number} $\WS(k,l)$ is the largest integer $n$ for which the set of the first $n$ positive integers $\{1,2,\ldots,n\}$ admits a $k$-partition into weakly $l$-sum-free sets.
\par For $l=2$, the numbers $\WS(k,2)$ are called \textit{weak Schur numbers}. The first few weak Schur numbers are given in Table~\ref{tab3}.
\begin{table}[!ht]
\begin{center}
\begin{tabular}{|c|c|c|c|c|c|c|}
\hline
$k$ & $1$ & $2$ & $3$ & $4$ & $5$ & $6$ \\
\hline
$\WS(k,2)$ & $2$ & $8$ & $23$ & $66$ & $\geq 196$ & $\geq 575$ \\
\hline
\end{tabular}
\end{center}
\caption{\label{tab3}The first few weak Schur numbers $\WS(k,2)$.}
\end{table}
\par The exact value of $\WS(4,2)$ was obtained by Blanchard, Harary and Reis \cite{Blanchard}. The lower bounds of $\WS(5,2)$ and $\WS(6,2)$ are due to Eliahou, Mar\'{i}n, Revuelta and Sanz \cite{Eliahou}.
\par More generally, the generalized weak Schur numbers are always finite (see \cite{Sierpinski,Irving,Landman} for instance). Moreover, the generalized weak Schur numbers appear as a good upper bound for the generalized Schur numbers, since a weakly $l$-sum-free set of integers is also $l$-sum-free. Therefore, we have
\begin{equation}\label{eq3}
\Sc(k,l)\leq \WS(k,l),
\end{equation}
for all positive integers $k$ and $l$. A trivial lower bound for the weak Schur numbers is
\begin{equation}\label{eq4}
kl\leq\WS(k,l),
\end{equation}
because each of the $k$ weakly sum-free sets can contain $l$ distinct integers without solution of the equation $x_1+\ldots+x_l=y$. Better lower bounds for $\WS(k,l)$ can be found in \cite{Sanz}.
\par A set of integers is said to be \textit{weakly $l$-sum-free modulo $m$} if it contains no pairwise distinct elements $x_1,\ldots,x_l,y$ satisfying
$$
x_1+\ldots+x_l\equiv y\pmod{m}.
$$
For every positive integer $k$, the \textit{generalized weak Schur number modulo $m$}, denoted by $\WS_m(k,l)$, is the largest integer $n$ for which the set of the first $n$ positive integers $\{1,2,\ldots,n\}$ admits a $k$-partition into weakly $l$-sum-free sets modulo $m$.
\par For every modulus $m$, the inequality
\begin{equation}\label{eq5}
\WS_m(k,l)\leq \WS(k,l)
\end{equation}
holds because a weakly $l$-sum-free set modulo $m$ of integers is also weakly $l$-sum-free.
\par Abbott and Wang conjectured that $\T(k)$ is equal to $\Sc(k,2)$, for all positive integer $k$. Here, in the weak case, it appears that considering similar numbers than $\T(k)$ is without great interest. Indeed, as we can see in Table~\ref{tab1} for $m\in\{1,\ldots,15\}$, the values of modular generalized weak Schur numbers, for $k=2$ and $l=2$, seem to be very difficult to predict. For $m\geq16$, we have $\WS_m(2,2)=\WS(2,2)$ because $\WS_m(2,2)\leq \WS(2,2) = 8$ and because, for two distinct integers $x,y\in\{1,\ldots,8\}$, we always have $x+y\leq 15 < m$.

\begin{table}[!ht]
\begin{center}
\begin{tabular}{|c|c|c|c|c|c|c|c|c|c|c|c|c|c|c|c|}
\hline
$m$ & $1$ & $2$ & $3$ & $4$ & $5$ & $6$ & $7$ & $8$ & $9$ & $10$ & $11$ & $12$ & $13$ & $14$ & $15$ \\
\hline
$\WS_m(2,2)$ & $4$ & $5$ & $4$ & $5$ & $6$ & $6$ & $7$ & $6$ & $7$ & $7$ & $7$ & $8$ & $8$ & $8$ & $8$ \\
\hline
\end{tabular}
\end{center}
\caption{\label{tab1}The modular generalized weak Schur numbers $\WS_m(2,2)$}
\end{table}

\par In this paper, we explicitly determine the modular generalized weak Schur numbers $\WS_m(k,l)$ for small values of $m$: for all moduli $m\in\{1,2,3\}$. For $m=1$, we obtain that
$$
\WS_1(k,l) = kl,
$$
for all $k\geq1$ and $l\geq1$, since a weakly $l$-sum-free set modulo $1$ has cardinality of at most $l$ because every $l+1$ distinct positive integers $x_1,\ldots,x_l,y$ verify $x_1+\ldots+x_l\equiv y\bmod{1}$. For $m=2$ and $m=3$ the following theorems will be proved in Section~\ref{secw2} and Section~\ref{secw3} respectively.

\begin{theorem}\label{thm3}
Let $k$ and $l$ be two positive integers. Then,
$$
\WS_2(k,l) = \left\{\begin{array}{lll}
l+1 & \text{for} & k=1\ \text{and}\ l\equiv 0,1\bmod{4},\\
l & \text{for} & k=1\ \text{and}\ l\equiv 2,3\bmod{4},\\
2(k-1)l+1 & \text{for} & k\geq2\ \text{and}\ l\ \text{even},\\
k(l+1) & \text{for} & \left\{\begin{array}{l}
k\geq2\ \text{and}\ l\equiv1\bmod{4},\\
k\geq2\ \text{even\ and}\ l\equiv3\bmod{4},\\
\end{array}\right.\\
k(l+1)-1 & \text{for} & k\geq3\ \text{odd\ and}\ l\equiv3\bmod{4}.
\end{array}\right.
$$
\end{theorem}

\begin{theorem}\label{thm4}
Let $k$ and $l$ be two positive integers. Then,
$$
\WS_3(k,l) = \left\{\begin{array}{lll}
3k & \text{for} & k\geq1\ \text{and}\ l=1,\\
l & \text{for} & k=1\ \text{and}\ l\geq2,\\
2l+2 & \text{for} & k=2\ \text{and}\ l\geq2,\ l\equiv0,1,5\bmod{9},\\
2l+1 & \text{for} & k=2\ \text{and}\ \left\{\begin{array}{l}
l=3,\\
l\geq5,\ l\equiv2,3,4,6,7,8\bmod{9},\\
\end{array}\right.\\
2l & \text{for} & k=2\ \text{and}\ l\in\{2,4\},\\
3(k-2)l+2 & \text{for} & k\geq3\ \text{and}\ l\equiv0,2\bmod{3},\\
k(l+1) & \text{for} & \left\{\begin{array}{l}
k=3,\ k\geq5\ \text{and}\ l\geq2,\ l\equiv1\bmod{3},\\
k=4\ \text{and}\ l\geq2,\ l\equiv1,7\bmod{9},\\
\end{array}\right.\\
4l+3 & \text{for} & k=4\ \text{and}\ l\equiv4\bmod{9}.\\
\end{array}\right.
$$
\end{theorem}

\subsection{Contents}

This paper is organized as follows. A simple proof of Theorems~\ref{thm1} and \ref{thm2} is given in Section~2, which completely determines the exact value of the generalized Schur numbers modulo $2$ and $3$. In Section~3, several basic results on the projection of partitions of integers in $\Zm$, which will be useful in our proofs in the sequel, are introduced. The generalized weak Schur numbers modulo $2$ are obtained by proving Theorem~\ref{thm3} in Section~4 and the case $m=3$ is settled in Section~5 by proving Theorem~\ref{thm4}. Remark that Theorems~\ref{thm3} and \ref{thm4} are the most difficult results to prove in this paper, much more difficult than Theorems~\ref{thm1} and \ref{thm2} whose proofs, in Section~2, are direct. Finally, in Section~6, we discuss open problems.

\section{\texorpdfstring{$\Sc_m(k,l)$}{Sm(k,l)} for the moduli \texorpdfstring{$m=2$}{m=2} and \texorpdfstring{$m=3$}{m=3}}\label{secsm}

We begin this section by proving Theorem~\ref{thm1}, that is,
$$
\Sc_2(k,l) = \left\{\begin{array}{ll}
0 & \text{for}\ k\geq1\ \text{and}\ l\ \text{odd},\\
1 & \text{for}\ k\geq1\ \text{and}\ l\ \text{even}.\\
\end{array}\right.
$$

\begin{proof}[Proof of Theorem~\ref{thm1}]
Let $k$ and $l$ be two positive integers. As already remarked in Section~1, the inequality $\Sc_2(k,l)\leq 1$ holds. Moreover, since $\sum_{i=1}^{l}1=l$, it follows that the integer $1$ belongs to a sum-free set of integers modulo $2$ if and only if $l$ is even.
\end{proof}

We end with the proof of Theorem~\ref{thm2}, that is,
$$
\Sc_3(k,l) = \left\{\begin{array}{ll}
0 & \text{for}\ k\geq1\ \text{and}\ l\equiv 1 \bmod{3},\\
1 & \text{for}\ k=1\ \text{and}\ l\equiv 0,2 \bmod{3},\\
2 & \text{for}\ k\geq2\ \text{and}\ l\equiv 0,2 \bmod{3}.
\end{array}\right.
$$

\begin{proof}[Proof of Theorem~\ref{thm2}]
Let $k$ and $l$ be two positive integers. Since $\sum_{i=1}^{l}1=l$, it follows that the integer $1$ belongs to a $l$-sum-free set of integers modulo $3$ if and only if $l\equiv 0$ or $2\bmod{3}$. Thus, we have  $\Sc_3(k,l)=0$ when $l\equiv 1 \bmod{3}$. Since $\sum_{i=1}^{l}1=l\equiv 2\bmod{3}$ for $l\equiv 2\bmod{3}$ and $\sum_{i=1}^{l-1}1+2=l+1\equiv 1\bmod{3}$ for $l\equiv 0\bmod{3}$, it follows that the integers $1$ and $2$ cannot belong together to a $l$-sum-free set of integers modulo $3$ when $l\equiv 0$ or $2\bmod{3}$. Therefore $\Sc_3(1,l)=1$ for $k=1$ when $l\equiv 0$ or $2\bmod{3}$. Finally, if $k\geq2$ and $l\equiv 0$ or $2\bmod{3}$, then $\Sc_3(k,l)\leq 2$ as remarked in Section~1 and $\Sc_3(k,l)\geq 2$ because the sets $\{1\}$ and $\{2\}$ are both $l$-sum-free modulo $3$. This leads to the formula $\Sc_3(k,l)=2$ in this case.
\end{proof}

\section{Projective partitions in \texorpdfstring{$\Zm$}{Z/mZ}}

Throughout this paper, projections of partitions into $\Zm$ will be considered. Let
$$
\pi_m : \Z \longrightarrow \Zm
$$
denote the canonical projection map. Let $S$ be a set of integers and let $P=\{S_1,\ldots,S_k\}$ be a $k$-partition of $S$. Denote by $\pi_m(P)$ the projection of the partition $P$ of $S$, that is the partition $\pi_m(P)=\{\pi_m(S_1),\ldots,\pi_m(S_k)\}$ of the multiset $\pi_m(S)$ of $\Zm$.
\par For example, if we consider the partition
$$
P=\{\{1,2,4,8\},\{3,5,6,7\}\}
$$
of the set $S=\{1,\ldots,8\}$, then we obtain that
$$
\pi_2(P) = \{\{1,0,0,0\},\{1,1,0,1\}\}
$$
is a partition of the multiset $\pi_2(S)=\{0,0,0,0,1,1,1,1\}$ of $\Zd$.
\par As for sets of integers, a multiset of $\Zm$ is said to be (weakly) $l$-sum-free if it contains no (pairwise distinct) elements $x_1,\ldots,x_l,y$ satisfying $x_1+\ldots+x_m=y$ in $\Zm$.

\begin{proposition}
A set $S$ of integers is (weakly) $l$-sum-free modulo $m$ if and only if its projection $\pi_m(S)$ is a (weakly) $l$-sum-free multiset of $\Zm$.
\end{proposition}

Remark that the partition $P$, of the previous example, is a partition of $S$ into weakly $2$-sum-free sets but it is not a partition of $S$ into weakly $2$-sum-free sets modulo $2$ since $\pi_2(P)$ is not a partition of $\pi_2(S)$ into weakly $2$-sum-free multisets in $\Zd$. Indeed, in each element of $\pi_2(P)$, we have three distinct elements satisfying either $0+0=0$ or $1+1=0$ in $\Zd$.

For every multiset $M$ of $\Zm$, denote by
$$
\m_M : \Zm \longrightarrow \N
$$
the multiplicity function associated with $M$, that is the function which assigns to each element $x\in\Zm$ its multiplicity in $M$. Let $|M|$ denote the cardinality of the multiset $M$, that is the number of elements constituting $M$, counted with multiplicity, that is
$$
|M|=\sum_{x\in\Zm}\m_M(x)\in\N.
$$

Obviously, by $\pi_n$, each $k$-partition of a set of $n$ positive integers can be associated uniquely with a $k$-partition of a multiset of $n$ terms in $\Zm$, counted with multiplicity. Conversely, this process is not bijective in general. Indeed, distinct partitions of the same set of integers can be projected on the same partition in $\Zm$.

\begin{proposition}
Let $m$ and $n$ be two positive integers and let $S=\{1,2,\ldots,n\}$. The exact number of $k$-partitions $P=\{S_1,\ldots,S_k\}$ of $S$ that have the same projective $k$-partition $\pi_m(P)=\{\pi_m(S_1),\ldots, \pi_m(S_k)\}$ is equal to
$$
\prod_{v=0}^{m-1} \prod_{u=1}^{k} \binom{\sum_{w=u}^{k}\m_{\pi_m(S_w)}(v)}{\m_{\pi_m(S_u)}(v)} = \prod_{v=0}^{m-1} \binom{\sum_{w=1}^{k}\m_{\pi_m(S_w)}(v)}{\m_{\pi_m(S_1)}(v),\m_{\pi_m(S_2)}(v),\ldots,\m_{\pi_m(S_k)}(v)},
$$
where $\binom{a}{b_1,b_2,\ldots,b_k}$ is the multinomial coefficient $\frac{a!}{{b_1}!{b_2}!\cdots{b_k}!}$.
\end{proposition}

\begin{proof}
Consider the euclidean division of $n$ by $m$, that is $n=qm+r$. Let $\varepsilon$ be the function defined by $\varepsilon : \Zm \longrightarrow \{0,1\}$ where $\varepsilon(v)=1$ for $v\in\{1,\ldots,r\}$ and $\varepsilon(v)=0$ for $v\in\{0,r+1,\ldots,m-1\}$. Thus there is exactly $q+\varepsilon(v)$ integers in $\{1,\ldots,n\}$ whose residue class modulo $m$ is $v\in\Zm$. We proceed by induction on $u$. Suppose that we have already chosen the integers in the first $u-1$ sets of the $k$-partition $P$. For the set $S_u$, we have to choose, for every $v\in\Zm$, $\m_{\pi_m(S_u)}(v)$ integers among the remaining $q+\varepsilon(v)-\sum_{w=1}^{u-1}\m_{\pi_m(S_w)}(v)$ integers whose residue class modulo $m$ is $v$. This corresponds to the binomial coefficient
$$
\binom{q+\varepsilon(v)-\sum_{w=1}^{u-1}\m_{\pi_m(S_w)}(v)}{\m_{\pi_m(S_u)}(v)} = \binom{\sum_{w=u}^{k}\m_{\pi_m(S_w)}(v)}{\m_{\pi_m(S_u)}(v)}.
$$
This completes the proof.
\end{proof}

For example, the number of $2$-partitions $P$ of $\{1,\ldots,8\}$ whose projection in $\Zd$ is the partition $\pi_2(P)=\{\{1,0,0,0\},\{1,1,1,0\}\}$ is equal to $16$, since $(\binom{4}{3}\cdot\binom{4}{1})\cdot(\binom{4-3}{1}\cdot\binom{4-1}{3})=4\cdot4=16$. These $16$ partitions are given below.
$$
\begin{array}{ccc}
\{\{1,2,4,6\},\{3,5,7,8\}\}, & \{\{1,2,4,8\},\{3,5,7,6\}\}, & \{\{1,2,6,8\},\{3,5,7,4\}\}, \\
\{\{1,4,6,8\},\{3,5,7,2\}\}, & \{\{3,2,4,6\},\{1,5,7,8\}\}, & \{\{3,2,4,8\},\{1,5,7,6\}\}, \\
\{\{3,2,6,8\},\{1,5,7,4\}\}, & \{\{3,4,6,8\},\{1,5,7,2\}\}, & \{\{5,2,4,6\},\{1,3,7,8\}\}, \\
\{\{5,2,4,8\},\{1,3,7,6\}\}, & \{\{5,2,6,8\},\{1,3,7,4\}\}, & \{\{5,4,6,8\},\{1,3,7,2\}\}, \\
\{\{7,2,4,6\},\{1,3,5,8\}\}, & \{\{7,2,4,8\},\{1,3,5,6\}\}, & \{\{7,2,6,8\},\{1,3,5,4\}\}, \\
\{\{7,4,6,8\},\{1,3,5,2\}\}. \\
\end{array}
$$

\section{\texorpdfstring{$\WS_m(k,l)$}{WSm(k,l)} for the modulus \texorpdfstring{$m=2$}{m=2}}\label{secw2}

The goal of this section is to prove Theorem~\ref{thm3}, that is,
$$
\WS_2(k,l) = \left\{\begin{array}{ll}
l+1 & \text{for}\ k=1\ \text{and}\ l\equiv 0,1 \bmod{4},\\
l & \text{for}\ k=1\ \text{and}\ l\equiv 2,3 \bmod{4},\\
2(k-1)l+1 & \text{for}\ k\geq2\ \text{and}\ l\ \text{even},\\
k(l+1) & \text{for}\ \left\{\begin{array}{l}
k\geq2\ \text{and}\ l\equiv 1 \bmod{4},\\
k\geq2\ \text{even\ and}\ l\equiv 3 \bmod{4},\\
\end{array}\right.\\
k(l+1)-1 & \text{for}\ k\geq3\ \text{odd\ and}\ l\equiv 3 \bmod{4}.
\end{array}\right.
$$

The proof of Theorem~\ref{thm3} is based on the following two lemmas.

\begin{lemma}\label{lem1}
For any weakly $l$-sum-free multiset $M$ of $\Zd$, with $|M|\geq1$, we have
\begin{enumerate}
\item
$\m_M(0)\leq l$.
\item
$\m_M(1)\leq l$ if $l$ is odd.
\item
$|M|\leq l+1$ if $|M\cap\Zd|=2$.
\end{enumerate}
\end{lemma}

\begin{proof}
\begin{enumerate}
\item[]
\item
Since $\sum_{i=1}^{l}0=0$, the multiset $M$ cannot contain more than $l$ terms $0\in\Zd$.
\item
Suppose that $l$ is odd, that is $\pi_2(l)=1$. Since $\sum_{i=1}^{l}1=\pi_2(l)=1$, the multiset $M$ cannot contain more than $l$ elements $1\in\Zd$.
\item
Suppose that $M=\{0,1,x_1,\ldots,x_l\}$ is a weakly $l$-sum-free multiset of $\Zd$ with cardinality $|M|=l+2$. Since the sum of the $l$ elements $\sum_{i=1}^{l}x_i$ is equal to either $0$ or $1$, which are both contained in $M\setminus\{x_1,\ldots,x_l\}$, we obtain a contradiction.
\end{enumerate}
\end{proof}

\begin{remark}
For $l$ even, a multiset only constituted by elements $1\in\Zd$, with any cardinality $|M|\geq1$, is always weakly $l$-sum free in $\Zd$ since $\sum_{i=1}^{l}1=\pi_2(l)=0$.
\end{remark}

\begin{lemma}\label{lem2}
Let $M=\{x_1,\ldots,x_{l+1}\}$ be a multiset of $\Zd$ with cardinality $|M|=l+1$. Then, $M$ is weakly $l$-sum-free if and only if $\sum_{i=1}^{l+1}x_i=1$.
\end{lemma}

\begin{proof}
First, by definition, the multiset $M$ is weakly $l$-sum-free if and only if, for every $1\leq j\leq l+1$, the inequality $\sum_{i=1,i\neq j}^{l+1}x_i\neq x_j$ holds. Moreover, since
$$
\sum_{\stackrel{i=1}{i\neq j}}^{l+1}x_i \neq x_j \ \text{for\ all}\ 1\leq j\leq l+1 \quad \begin{array}[t]{l} \displaystyle\Longleftrightarrow \quad \sum_{i=1}^{l+1}x_i \neq 2x_j = 0 \ \text{for\ all}\ 1\leq j\leq l+1 \\[3ex] \displaystyle\Longleftrightarrow \quad \sum_{i=1}^{l+1}x_i=1,\end{array}
$$
the result follows.
\end{proof}

We are now ready to prove Theorem~\ref{thm3}, the main result of this section.

\subsection{For \texorpdfstring{$k=1$}{k=1}}\label{m2k1}

Let $k=1$ and let $l$ be a positive integer. We will prove that
$$
\WS_2(1,l) = \left\{\begin{array}{ll}
l+1 & \text{for}\ l\equiv 0,1 \bmod{4},\\
l & \text{for}\ l\equiv 2,3 \bmod{4}.\\
\end{array}\right.
$$

\begin{claim}
$\WS_2(1,l) \in \{ l, l+1\}$ for all positive integers $l$.
\end{claim}

\begin{proof}
Since a multiset of $l$ elements in $\Zd$ is always weakly $l$-sum-free, the inequality $\WS_2(1,l)\geq l$ holds. Moreover, from Lemma~\ref{lem1}, we know that a weakly $l$-sum-free multiset $M$ such that $|M\cap\Zd|=2$ has cardinality of at most $l+1$. Therefore $\WS_2(1,l)\leq l+1$.
\end{proof}

\begin{claim}
Let $S=\{1,2,\ldots,l+1\}$. Then, the multiset $\pi_2(S)$ is weakly $l$-sum-free if and only if $l\equiv 0$ or $1\bmod{4}$.
\end{claim}

\begin{proof}
By Lemma~\ref{lem2}, the multiset $\pi_2(S)$ is weakly $l$-sum-free if and only if $\sum_{x\in S}x \equiv 1 \bmod{2}$. Moreover, since
$$
\sum_{x\in S}x = \sum_{x=1}^{l+1}x = \frac{(l+1)(l+2)}{2} \equiv 1 \pmod{2}
$$
if and only if $l\equiv 0$ or $1\bmod{4}$, the result follows.
\end{proof}

\subsection{For \texorpdfstring{$k\geq2$}{k>=2} and \texorpdfstring{$l$}{l} even}

Let $k\geq 2$ and $l$ be two positive integers, with $l$ even. We will prove that
$$
\WS_2(k,l)=2(k-1)l+1.
$$

\begin{claim}
$\WS_2(k,l) \leq 2(k-1)l+1$ for $k\geq 2$ and $l$ even.
\end{claim}

\begin{proof}
Let $S=\{1,2,\ldots,2(k-1)l+2\}$. Suppose that there exists a partition $P=\{S_1,\ldots,S_k\}$ of the multiset $\pi_2(S)$ into $k$ weakly $l$-sum-free multisets of $\Zd$. So the multiplicity function of $\pi_2(S)$ is defined by $\m_{\pi_2(S)}(0)=\m_{\pi_2(S)}(1)=(k-1)l+1$. Since $\m_{S_i}(0)\leq l$ for every $1\leq i\leq k$ by Lemma~\ref{lem1} and $\sum_{i=1}^{k}\m_{S_i}(0)=\m_{\pi_2(S)}(0)=l(k-1)+1$, it follows that $\m_{S_i}(0)\geq 1$ for every $1\leq i\leq k$ by the pigeonhole principle. It follows that $|S_i|\leq l+1$ for every $1\leq i\leq k$ by Lemma~\ref{lem1} again and thus we obtain the following upper bound of the cardinality of $\pi_2(S)$,
$$
|\pi_2(S)| = \sum_{i=1}^{k}|S_i| \leq k(l+1).
$$
Since
$$
(2(k-1)l+2)-k(l+1) = 2kl-2l+2-kl-k = k(l-1)-2(l-1) = (k-2)(l-1),
$$
we obtain that $2(k-1)l+2>k(l+1)$ for all $l\geq2$ and $k\geq3$, in contradiction with the previous inequality. For $k=2$, we have $\m_{\pi_2(S)}(0)=\m_{\pi_2(S)}(1)=l+1$ and $|S_1|=|S_2|=l+1$. It follows from Lemma~\ref{lem2} that $\sum_{x_1\in S_1}x_1=\sum_{x_2\in S_2}x_2=1$ and the contradiction comes from the following equality
$$
0 = 1+1 = \sum_{x_1\in S_1}x_1 + \sum_{x_2\in S_2}x_2 = \sum_{x\in \pi_2(S)}x = \pi_2\left(\sum_{x=1}^{2l+2}x\right) = \pi_2((l+1)(2l+3)) \stackrel{l\ \text{even}}{=} 1.
$$
This completes the proof.
\end{proof}

\begin{claim}
$\WS_2(k,l) \geq 2(k-1)l+1$ for $k\geq 2$ and $l$ even.
\end{claim}

\begin{proof}
Let $S=\{1,2,\ldots,2(k-1)l+1\}$. We will exhibit a partition $P=\{S_1,\ldots,S_k\}$ of $\pi_2(S)$ into $k$ weakly $l$-sum-free multisets of $\Zd$. Only the multiplicities of the elements constituting the multisets $S_i$ are reported here.
\begin{center}
\begin{tabular}{|c|c|c|c|}
\hline
$M$ & $\m_{M}(0)$ & $\m_{M}(1)$ & $|M|$ \\
\hline\hline
$S_1$ & $0$ & $(k-1)l+1$ & $(k-1)l+1$ \\
\hline
$S_2,\ldots,S_k$ & $l$ & $0$ & $l$ \\
\hline\hline
$\pi_2(S)$ & $(k-1)l$ & $(k-1)l+1$ & $2(k-1)l+1$ \\
\hline
\end{tabular}
\end{center}
First, $P$ is a partition of $\pi_2(S)$ since the multiplicity functions verify that $\sum_{i=1}^{k}\m_{S_i}(0)=\m_{\pi_2(S)}(0)=l(k-1)$ and $\sum_{i=1}^{k}\m_{S_i}(1)=\m_{\pi_2(S)}(1)=l(k-1)+1$. The multiset $S_1$ is weakly $l$-sum-free because, as already remarked above, a multiset which is only constituted by elements $1\in\Zd$ is always weakly $l$-sum-free when $l$ is even. For the other multisets $S_2,\ldots,S_k$, we already know that multisets containing only $l$ elements are always weakly $l$-sum-free. This completes the proof.
\end{proof}

\subsection{For \texorpdfstring{$k\geq2$}{k>=2} and \texorpdfstring{$l$}{l} odd}

Let $k\geq 2$ and $l$ be two positive integers, with $l$ odd. We will prove that
$$
\WS_2(k,l) = \left\{
\begin{array}{ll}
k(l+1) & \text{for} \ \left\{
\begin{array}{l}
k\ \text{even\ and}\ l\ \text{odd},\\
k\ \text{odd\ and}\ l\equiv1\bmod{4},
\end{array}
\right.\\
k(l+1)-1 & \text{for}\ k\ \text{odd\ and}\ l\equiv3\bmod{4}.
\end{array}
\right.
$$

\begin{claim}\label{claim3}
$\WS_2(k,l)\leq k(l+1)$ for $k\geq1$ and $l$ odd.
\end{claim}

\begin{proof}
Directly follows from Lemma~\ref{lem1}.
\end{proof}

\begin{claim}\label{claim2}
$\WS_2(k,l)\leq k(l+1)-1$ for $k$ odd and $l\equiv3\bmod{4}$.
\end{claim}

\begin{proof}
Let $S=\{1,2,\ldots,k(l+1)\}$. Suppose that there exists a partition $P=\{S_1,\ldots,S_k\}$ of $\pi_2(S)$ into $k$ weakly $l$-sum-free multisets of $\Zd$. For every $1\leq i\leq k$, since $S_i$ is a weakly $l$-sum-free multiset and since $l$ is odd, we know from Lemma~\ref{lem1} that $|S_i|\leq l+1$. Moreover, we have $\sum_{i=1}^{k}|S_i|=|\pi_2(S)|=k(l+1)$. Therefore, for all $1\leq i\leq k$, the multiset $S_i$ has cardinality of $|\pi_2(S_i)|=l+1$ and thus $\sum_{x\in S_i}x=1$ by Lemma~\ref{lem2}. Since $l\equiv3\bmod{4}$, this leads to
$$
\pi_2(k) = \sum_{i=1}^{k}1 = \sum_{i=1}^{k}\sum_{x\in S_i}x = \sum_{x\in\pi_2(S)}x = \pi_2\left(\sum_{i=1}^{k(l+1)}i\right) = \pi_2\left(\frac{l+1}{2}k(k(l+1)+1)\right) = 0,
$$
in contradiction with the hypothesis that $k$ is odd. This completes the proof.
\end{proof}

\begin{claim}\label{claim1}
$\WS_2(k,l)\geq \WS_2(k-2,l)+2(l+1)$ for $k\geq3$ and $l$ odd.
\end{claim}

\begin{proof}
Let $S=\{1,2,\ldots,\WS_2(k-2,l)\}$ and let $P=\{S_1,\ldots,S_{k-2}\}$ be a partition of $\pi_2(S)$ into $k$ weakly $l$-sum-free multisets of $\Zd$. Consider the multisets $S_{k-1}$ and $S_k$ of $\Zd$ defined below by their multiplicity functions.
\begin{center}
\begin{tabular}{|c|c|c|c|}
\hline
$M$ & $\m_M(0)$ & $\m_M(1)$ & $|M|$ \\
\hline\hline
$S_{k-1}$ & $l$ & $1$ & $l+1$ \\
\hline
$S_{k}$ & $1$ & $l$ & $l+1$ \\
\hline\hline
$S_{k-1}\cup S_{k}$ & $l+1$ & $l+1$ & $2(l+1)$ \\
\hline
\end{tabular}
\end{center}
Then $P'=\{S_1,\ldots,S_k\}$ is a partition of the multiset $\pi_2(S')=\pi_2(\{1,2,\ldots,\WS_2(k-2,l)+2(l+1)\})$ since $\sum_{i=1}^{k}\m_{S_i}(x)=\m_{\pi_2(S)}(x)+(l+1)=\m_{\pi_2(S')}(x)$ for all $x\in\Zd$. Moreover, since $l$ is odd, it follows that $\sum_{x\in S_{k-1}}x = \sum_{x\in S_{k}}=1$ and thus the multisets $S_{k-1}$ and $S_k$ are weakly $l$-sum-free in $\Zd$ by Lemma~\ref{lem2}. Therefore $P'$ is a partition of $\pi_2(S')$ into $k$ weakly $l$-sum-free multisets.
\end{proof}

\begin{claim}\label{claim11}
$\WS_2(2,l) \geq 2(l+1)$ for $l$ odd.
\end{claim}

\begin{proof}
Let $S=\{1,2,\ldots,2(l+1)\}$. As already seen in the proof of Claim~\ref{claim1}, the following partition $P=\{S_1,S_2\}$ of $\pi_2(S)$ is weakly $l$-sum-free.
\begin{center}
\begin{tabular}{|c|c|c|c|}
\hline
$M$ & $\m_M(0)$ & $\m_M(1)$ & $|M|$ \\
\hline\hline
$S_{1}$ & $l$ & $1$ & $l+1$ \\
\hline
$S_{2}$ & $1$ & $l$ & $l+1$ \\
\hline\hline
$S_{1}\cup S_{2}$ & $l+1$ & $l+1$ & $2(l+1)$ \\
\hline
\end{tabular}
\end{center}
This concludes the proof.
\end{proof}

We are now ready to prove the formula for all $k\geq 2$ and $l$ odd. We distinguish different cases depending on the parity of $k$ and the residue class of $l$ modulo $4$.\\[2ex]
\textbf{Case 1:} for $k\geq2$ odd and $l\equiv 1\bmod{4}$.\\
First, we know that $\WS_2(1,l)=l+1$ from Subsection~\ref{m2k1}. By applying $(k-1)/2$ times the inequality of Claim~\ref{claim1}, we obtain that
$$
\WS_2(k,l) \geq \WS_2(k-2,l) + 2(l+1) \geq \cdots \geq \WS_2(1,l) + (k-1)(l+1) = k(l+1).
$$
Finally, since $\WS_2(k,l)\leq k(l+1)$ by Claim~\ref{claim3}, it follows that $\WS_2(k,l)=k(l+1)$ in this case.\\[2ex]
\textbf{Case 2:} for $k\geq2$ odd and $l\equiv 3\bmod{4}$.\\
First, we know that $\WS_2(1,l)=l$ from Subsection~\ref{m2k1}. By applying $(k-1)/2$ times the inequality of Claim~\ref{claim1}, we obtain that
$$
\WS_2(k,l) \geq \WS_2(k-2,l) + 2(l+1) \geq \cdots \geq \WS_2(1,l) + (k-1)(l+1) = k(l+1)-1.
$$
Finally, since $\WS_2(k,l)\leq k(l+1)-1$ by Claim~\ref{claim2}, it follows that $\WS_2(k,l)=k(l+1)-1$ in this case.\\[2ex]
\textbf{Case 3:} for $k\geq2$ even and $l$ odd.\\
First, we know that $\WS_2(2,l)\geq 2(l+1)$ from Claim~\ref{claim11}. By applying $k/2-1$ times the inequality of Claim~\ref{claim1}, we obtain that
$$
\WS_2(k,l) \geq \WS_2(k-2,l) + 2(l+1) \geq \cdots \geq \WS_2(2,l) + (k-2)(l+1) \geq k(l+1),
$$
Finally, since $\WS_2(k,l)\leq k(l+1)$ by Claim~\ref{claim3}, it follows that $\WS_2(k,l)=k(l+1)$ in this case.\\[2ex]
This concludes the proof of Theorem~\ref{thm3}.

\section{\texorpdfstring{$\WS_m(k,l)$}{WSm(k,l)} for the modulus \texorpdfstring{$m=3$}{m=3}}\label{secw3}

The goal of this section is to prove Theorem~\ref{thm4}, that is,
$$
\WS_3(k,l) = \left\{\begin{array}{ll}
3k & \text{for}\ k\geq1\ \text{and}\ l=1,\\
l & \text{for}\ k=1\ \text{and}\ l\geq2,\\
2l+2 & \text{for}\ k=2\ \text{and}\ l\geq2,\ l\equiv 0,1,5 \bmod{9},\\
2l+1 & \text{for}\ k=2\ \text{and}\ \left\{\begin{array}{l}
l=3,\\
l\geq 5,\ l\equiv 2,3,4,6,7,8 \bmod{9},\\
\end{array}\right.\\
2l & \text{for}\ k=2\ \text{and}\ l\in\{2,4\},\\
3(k-2)l+2 & \text{for}\ k\geq3\ \text{and}\ l\equiv 0,2 \bmod{3},\\
k(l+1) & \text{for}\ \left\{\begin{array}{l}
k=3,\ k\geq 5\ \text{and}\ l\geq 2,\ l\equiv 1 \bmod{3},\\
k=4\ \text{and}\ l\geq 2,\ l\equiv 1,7 \bmod{9},\\
\end{array}\right.\\
4l+3 & \text{for}\ k=4\ \text{and}\ l\equiv 4 \bmod{9}.\\
\end{array}\right.
$$

The proof of Theorem~\ref{thm4} is based upon the following two lemmas.

\begin{lemma}\label{lem3}
For any weakly $l$-sum-free multiset $M$ of $\Zt$, with  $|M|\geq1$ and $l\geq2$, we have
\begin{enumerate}
\item
$\m_M(0)\leq l$.
\item
$\m_M(1)\leq l$ and $\m_M(2)\leq l$ if $l\equiv1\bmod{3}$.
\item
$|M|\leq l+1$ if $|M\cap\Zt|=2$.
\item
$|M|\leq l$ if $|M\cap\Zt|=3$.
\end{enumerate}
\end{lemma}

\begin{proof}
\begin{enumerate}
\item[]
\item
Since $\sum_{i=1}^{l}0=0$, the multiset $M$ cannot contain more than $l$ terms $0\in\Zt$.
\item
Suppose that $l\equiv1\bmod{3}$. Since $\sum_{i=1}^{l}1=\pi_3(l)=1$ and $\sum_{i=1}^{l}2=\pi_3(2l)=2$, the multiset $M$ can contain neither more than $l$ elements $1\in\Zt$ nor more than $l$ elements $2\in\Zt$.
\item
Supppose that $|M|=l+2$ with $M\cap\Zt=\{x,y\}$. Since $l+2\geq4$, we can suppose that $\m_M(x)\geq2$. Moreover, either $x=y+1$ or $x=y+2$. Without loss of generality, suppose that $x=y+1$ and denote $M=\{x,x,y,x_1,x_2,\ldots,x_{l-1}\}$. Then,
\begin{itemize}
\item
$x+\sum_{i=1}^{l-1}x_i=x$ if $\sum_{i=1}^{l-1}x_i=0$,
\item
$y+\sum_{i=1}^{l-1}x_i=x$ if $\sum_{i=1}^{l-1}x_i=1$,
\item
$x+\sum_{i=1}^{l-1}x_i=y$ if $\sum_{i=1}^{l-1}x_i=2$.
\end{itemize}
Therefore $M$ is not weakly $l$-sum-free in $\Zt$.
\item
Suppose that $|M|=l+1$ with $M\cap\Zt=\Zt$. Denote $M=\{0,1,2,x_1,\ldots,x_{l-2}\}$. Then,
\begin{itemize}
\item
$\sum_{i=1}^{l-2}x_i+1+2=0$ if $\sum_{i=1}^{l-2}x_i=0$,
\item
$\sum_{i=1}^{l-2}x_i+1+0=2$ if $\sum_{i=1}^{l-2}x_i=1$,
\item
$\sum_{i=1}^{l-2}x_i+2+0=1$ if $\sum_{i=1}^{l-2}x_i=2$.
\end{itemize}
Therefore $M$ is not weakly $l$-sum-free in $\Zt$.
\end{enumerate}
\end{proof}

\begin{lemma}\label{lem4}
Let $l\geq2$ and let $M=\{x_1,\ldots,x_{l+1}\}$ be a multiset of $\Zt$ with $|M|=l+1$ and $|M\cap\Zt|=2$. Let $x$ be the element of $\Zt$ such that $x\not\in M$. Then, the multiset $M$ is weakly $l$-sum-free if and only if $\sum_{i=1}^{l+1}x_i=2x$.
\end{lemma}

\begin{proof}
First, by definition, the multiset $M$ is weakly $l$-sum-free if and only if, for every $1\leq j\leq l+1$, the inequality $\sum_{i=1,i\neq j}^{l+1}x_i\neq x_j$ holds. Moreover, since
$$
\sum_{\stackrel{i=1}{i\neq j}}^{l+1}x_i \neq x_j \ \text{for\ all}\ 1\leq j\leq l+1 \quad \begin{array}[t]{l} \displaystyle\Longleftrightarrow \quad \sum_{i=1}^{l+1}x_i \neq 2x_j \ \text{for\ all}\ 1\leq j\leq l+1 \\[3ex] \displaystyle\Longleftrightarrow \quad \sum_{i=1}^{l+1}x_i=2x,\end{array}
$$
the result follows.
\end{proof}

We are now ready to prove Theorem~\ref{thm4}, the main result of this section.

\subsection{For \texorpdfstring{$k=1$}{k=1} and \texorpdfstring{$l=1$}{l=1}}

\begin{claim}
$\WS_3(k,1)=3k$ for $k\geq1$.
\end{claim}

\begin{proof}
Obviously, the inequality $\WS_3(k,1)\leq 3k$ holds because a weakly $1$-sum-free multiset of $\Zt$ cannot contain more than once each element of $\Zt$. Let $S=\{1,2,\ldots,3k\}$. The partition $P=\{S_1,\ldots,S_k\}$, where $S_i=\{0,1,2\}$ for all $1\leq i\leq k$, is a weakly $1$-sum-free partition of $\pi_3(S)$ and thus $\WS_3(k,1)\geq 3k$. This completes the proof.
\end{proof}

\begin{claim}\label{claim10}
$\WS_3(1,l)=l$ for $l\geq2$.
\end{claim}

\begin{proof}
First, since a multiset of $\Zt$ constituted by only $l$ terms is always weakly $l$-sum-free, it follows that $\WS_3(1,l)\geq l$. Moreover, by Lemma~\ref{lem3}, a weakly $l$-sum-free multiset $M$ such that $|M\cap\Zt|=3$ has cardinality of at most $l$. Therefore $\WS_3(1,l)\leq l$. This completes the proof.
\end{proof}

\subsection{For \texorpdfstring{$k=2$}{k=2}}

Let $k=2$ and let $l\ge2$ be a positive integer. We will prove that
$$
\WS_3(2,l) = \left\{\begin{array}{ll}
2l+2 & \text{for}\ l\ge2,\ l\equiv 0,1,5 \bmod{9},\\
2l+1 & \text{for}\ \left\{\begin{array}{l}
l=3,\\
l\geq5,\ l\equiv 2,3,4,6,7,8 \bmod{9},\\
\end{array}\right.\\
2l & \text{for}\ l\in\{2,4\}.
\end{array}\right.
$$

\begin{claim}\label{claim12}
Let $S=\{1,2,\ldots,n\}$ where $n\geq3$. Suppose that there exists a partition $P=\{S_1,S_2\}$ of $\pi_3(S)$ into $2$ weakly $l$-sum-free multisets. Then, the multisets $S_1$ and $S_2$ both have cardinality of at most $l+1$.
\end{claim}

\begin{proof}
Suppose, without loss of generality, that $|S_1|\geq l+2$. First, we deduce from Lemma~\ref{lem3} that $|S_1\cap\Zt|=1$. Let $x\in\Zt$ such that $S_1\cap\Zt=\{x\}$. Obviously, we have $\m_{\pi_3(S)}(x)\leq \m_{\pi_3(S)}(x+1)+\m_{\pi_3(S)}(x+2)$. This leads to the inequality
$$
|S_1| = \m_{S_1}(x) \leq \m_{S_1\cup S_2}(x) \leq \m_{S_1\cup S_2}(x+1) + \m_{S_1\cup S_2}(x+2) = \m_{S_2}(x+1) + \m_{S_2}(x+2) \leq |S_2|.
$$
It follows that $|S_2|\geq l+2$ and thus $|S_2\cap\Zt|=1$, by Lemma~\ref{lem3} again, in contradiction with $|\pi_3(S)\cap\Zt|=3$.
\end{proof}

\begin{claim}\label{claim13}
$\WS_3(2,l)\in\{2l,2l+1,2l+2\}$ for all positive integers $l\ge2$.
\end{claim}

\begin{proof}
Since a multiset of $l$ elements in $\Zt$ is always weakly $l$-sum-free, the inequality $\WS_3(2,l)\geq 2l$ holds. Moreover, from Claim~\ref{claim12}, we know that each multiset of a weakly $l$-sum-free $2$-partition has cardinality of at most $l+1$. Therefore $\WS_2(2,l)\leq 2(l+1)$.
\end{proof}

\begin{claim}\label{claim14}
$\WS_3(2,l)\geq 2l+2$ if and only if $l\equiv 0,1,5 \bmod{9}$.
\end{claim}

\begin{proof}
Let $S=\{1,2,\ldots,2l+2\}$. Suppose that there exists a partition $P=\{S_1,S_2\}$ of $\pi_3(S)$ into $2$ weakly $l$-sum-free multisets of $\Zt$. Since $|S_1|=|S_2|=l+1$ by Claim~\ref{claim12}, it follows from Lemma~\ref{lem3} that $|S_1\cap\Zt|=|S_2\cap\Zt|=2$.\\[2ex]
\textbf{Case 1:} if $0\in S_1$ and $0\in S_2$.\\
Without lost of generality, suppose that $S_1\cap\Zt=\{0,1\}$ and $S_2\cap\Zt=\{0,2\}$. By Lemma~\ref{lem4}, we know that the multisets $S_1$ and $S_2$ are weakly $l$-sum-free if and only if $\sum_{x\in S_1}x=2.2=1$ and $\sum_{x\in S_2}x=2.1=2$. It follows that $\m_{S_1}(1)\equiv 1\bmod{3}$ and $\m_{S_2}(2)\equiv 1\bmod{3}$. This implies that
$$
\m_{\pi_3(S)}(1)=\m_{S_1}(1)\equiv 1 \equiv \m_{S_2}(2) = \m_{\pi_3(S)}(2)\pmod{3}.
$$
Since $S$ is the set of the first $2l+2$ positive integers, it follows that $\m_{\pi_3(S)}(1)=\m_{\pi_3(S)}(2)$ or $\m_{\pi_3(S)}(1)=\m_{\pi_3(S)}(2)+1$. Therefore $\m_{\pi_3(S)}(1)=\m_{\pi_3(S)}(2)\equiv1\bmod{3}$. Moreover, by definition of the set $S$ again, either $\m_{\pi_3(S)}(0)=\m_{\pi_3(S)}(1)$ or $\m_{\pi_3(S)}(0)=\m_{\pi_3(S)}(1)-1$.\\[2ex]
\textbf{Case 1.1:} if $\m_{\pi_3(S)}(0)=\m_{\pi_3(S)}(1)=\m_{\pi_3(S)}(2)\equiv 1\bmod{3}$.\\
Then, $2l+2=\m_{\pi_3(S)}(0)+\m_{\pi_3(S)}(1)+\m_{\pi_3(S)}(2)\equiv 3 \bmod{9}$ and thus $l\equiv 5 \bmod{9}$. In this case, $l\equiv 5\bmod{9}$, we can verify with Lemma~\ref{lem4} that the following $2$-partition of $\pi_3(S)$ is weakly $l$-sum-free.
\begin{center}
\begin{tabular}{|c|c|c|c|c|}
\hline
$M$ & $\m_M(0)$ & $\m_M(1)$ & $\m_M(2)$ & $|M|$ \\
\hline\hline
$S_1$ & $(l+1)/3$ & $2(l+1)/3$ & $0$ & $l+1$ \\
\hline
$S_2$ & $(l+1)/3$ & $0$ & $2(l+1)/3$ & $l+1$ \\
\hline\hline
$\pi_3(S)$ & $2(l+2)/3$ & $2(l+2)/3$ & $2(l+2)/3$ & $2l+2$ \\
\hline
\end{tabular}
\end{center}
This proves that $\WS_3(2,l) \geq 2l+2$ for all $l\equiv 5\bmod{9}$.\\[2ex]
\textbf{Case 1.2:} if $\m_{\pi_3(S)}(0)+1=\m_{\pi_3(S)}(1)=\m_{\pi_3(S)}(2)\equiv 1\bmod{3}$.\\
Then, $2l+2=\m_{\pi_3(S)}(0)+\m_{\pi_3(S)}(1)+\m_{\pi_3(S)}(2)\equiv 2 \bmod{9}$ and thus $l\equiv 0 \bmod{9}$. We can verify with Lemma~\ref{lem4} that the following $2$-partition of $\pi_3(S)$ is weakly $l$-sum-free in this case.
\begin{center}
\begin{tabular}{|c|c|c|c|c|}
\hline
$M$ & $\m_M(0)$ & $\m_M(1)$ & $\m_M(2)$ & $|M|$ \\
\hline\hline
$S_1$ & $l/3$ & $2l/3+1$ & $0$ & $l+1$ \\
\hline
$S_2$ & $l/3$ & $0$ & $2l/3+1$ & $l+1$ \\
\hline\hline
$\pi_3(S)$ & $2l/3$ & $2l/3+1$ & $2l/3+1$ & $2l+2$ \\
\hline
\end{tabular}
\end{center}
This proves that $\WS_3(2,l) \geq 2l+2$ for all $l\equiv 0\bmod{9}$.\\[2ex]
\textbf{Case 2:} if only one of the multisets $S_1$ and $S_2$ contains elements $0\in\Zt$.\\
Without loss of generality, we can suppose that this is $S_1$. Then, we have $S_1\cap\Zt=\{0,a\}$ and $S_2\cap\Zt=\{a,2a\}$ with $a\in\{1,2\}\subset\Zt$. By Lemma~\ref{lem4}, we know that the multisets $S_1$ and $S_2$ are weakly $l$-sum-free if and only if $\sum_{x\in S_1}x=2.2a=a$ and $\sum_{x\in S_2}x=2.0=0$. It follows that $\m_{S_1}(a)\equiv 1\bmod{3}$ and $\m_{S_2}(a)\equiv \m_{S_2}(2a)\bmod{3}$. This implies that
$$
\m_{\pi_3(S)}(a) = \m_{S_1}(a) + \m_{S_2}(a) \equiv 1 + \m_{S_2}(2a) = 1 + \m_{\pi_3(S)}(2a) \pmod{3}.
$$
Since $\m_{\pi_3(S)}(1)=\m_{\pi_3(S)}(2)$ or $\m_{\pi_3(S)}(1)=\m_{\pi_3(S)}(2)+1$ by definition of the set $S$, we deduce that $a=1$ and $\m_{\pi_3(S)}(1)=\m_{\pi_3(S)}(2)+1$. Moreover, by definition of the set $S$ again, we have
$$
\m_{\pi_3(S)}(0)=\m_{\pi_3(S)}(2)=\m_{\pi_3(S)}(1)-1.
$$
Then $2l+2=\m_{\pi_3(S)}(0)+\m_{\pi_3(S)}(1)+\m_{\pi_3(S)}(2)=3\m_{\pi_3(S)}(0)+1\equiv1\bmod{3}$ and thus $l\equiv 1\bmod{3}$. It follows that $2\equiv l+1 = \m_{S_2}(1)+\m_{S_2}(2)\equiv 2.\m_{S_2}(2) = 2.\m_{\pi_3(S)}(2) \bmod{3}$. Therefore
$$
\m_{\pi_3(S)}(0)=\m_{\pi_3(S)}(2)=\m_{\pi_3(S)}(1)-1\equiv1\pmod{3}.
$$
Then $2l+2=\m_{\pi_{3}(S)}(0)+\m_{\pi_{3}(S)}(1)+\m_{\pi_{3}(S)}(2)\equiv 4\bmod{9}$ and thus $l\equiv 1\bmod{9}$. We can verify with Lemma~\ref{lem4} that the following $2$-partition of $\pi_3(S)$ is weakly $l$-sum-free in this case.
\begin{center}
\begin{tabular}{|c|c|c|c|c|}
\hline
$M$ & $\m_M(0)$ & $\m_M(1)$ & $\m_M(2)$ & $|M|$ \\
\hline\hline
$S_1$ & $(2l+1)/3$ & $(l+2)/3$ & $0$ & $l+1$ \\
\hline
$S_2$ & $0$ & $(l+2)/3$ & $(2l+1)/3$ & $l+1$ \\
\hline\hline
$\pi_3(S)$ & $(2l+1)/3$ & $(2l+4)/3$ & $(2l+1)/3$ & $2l+2$ \\
\hline
\end{tabular}
\end{center}
This proves that $\WS_3(2,l) \geq 2l+2$ for all $l\geq2$, $l\equiv 1\bmod{9}$.\\[2ex]
This completes the proof that $\WS_3(2,l)\geq 2l+2$ if and only if $l\equiv 0,1,5 \bmod{9}$.
\end{proof}

\begin{claim}\label{claim9}
$\WS_3(2,l)\geq 2l+1$ for $l\geq9$.
\end{claim}

\begin{proof}
Let $S=\{1,2,\ldots,2l+1\}$. We will prove that there exists a partition of $\pi_3(S)$ into $2$ weakly $l$-sum-free multisets. Consider the euclidean division of $l\geq9$ by $3$ : $l=3q+r$ where $q\geq3$ and $r\in\{0,1,2\}$. Then, $l+1=3q+r+1$ and $2l+1=6q+2r+1$ and $\pi_3(S)$ is a multiset of $\Zt$ whose multiplicity function verifies that
$$
\m_{\pi_3(S)}(1)\geq 2q+1,\quad \m_{\pi_3(S)}(2)\geq 2q\quad \text{and}\quad \m_{\pi_3(S)}(0)\geq 2q.
$$
First, we prove that there exists a submultiset $S_1$ with $|S_1|=l+1$ which is weakly $l$-sum-free. The multiset $S_1$ is defined below by its multiplicity function, distinguishing different cases depending on the residue class of $q$ modulo $3$.
\begin{center}
\begin{tabular}{|c|c|c|c|}
\hline
$\pi_3(q)$ & $\m_{S_1}(0)$ & $\m_{S_1}(1)$ & $\m_{S_1}(2)$ \\
\hline\hline
$0$ & $q+r$ & $2q+1$ & $0$ \\
\hline
$1$ & $q+r+2$ & $2q-1$ & $0$ \\
\hline
$2$ & $q+r+1$ & $2q$ & $0$ \\
\hline
\end{tabular}
\end{center}
\textbf{Case 1:} for $q\equiv0\bmod{3}$.\\
The multiset $S_1$ is well defined because $\m_{S_1}(0)=q+r \leq 2q \leq \m_{\pi_3(S)}(0)$ and $\m_{S_1}(1)=\m_{\pi_3(S)}(1)$. Moreover, it is weakly $l$-sum-free by Lemma~\ref{lem4} since
$$
\sum_{x\in S_1}x = (q+r).0 + (2q+1).1 = \pi_3(2q+1) = 1 = 2.2.
$$
\textbf{Case 2:} for $q\equiv1\bmod{3}$.\\
Remark that $q\geq4$ in this case. The multiset $S_1$ is well defined because $\m_{S_1}(0)=q+r+2 \leq 2q \leq \m_{\pi_3(S)}(0)$ and $\m_{S_1}(1)= 2q-1\leq 2q+1\leq \m_{\pi_3(S)}(1)$. Moreover, it is weakly $l$-sum-free by Lemma~\ref{lem4} since
$$
\sum_{x\in S_1}x = (q+r+2).0 + (2q-1).1 = \pi_3(2q-1) = 1 = 2.2.
$$
\textbf{Case 3:} for $q\equiv2\bmod{3}$.\\
The multiset $S_1$ is well defined because $\m_{S_1}(0)=q+r+1 \leq 2q \leq \m_{\pi_3(S)}(0)$ and $\m_{S_1}(1)=2q\leq 2q+1\m_{\pi_3(S)}(1)$. Moreover, it is weakly $l$-sum-free by Lemma~\ref{lem4} since
$$
\sum_{x\in S_1}x = (q+r+1).0 + (2q).1 = \pi_3(2q) = 1 = 2.2.
$$
Finally, if $S_2$ is the multiset, with cardinality $l$, constituted by all the other elements of $\pi_3(S)\setminus S_1$, then $S_2$ is clearly weakly $l$-sum-free and we have obtained a weakly $l$-sum-free $2$-partition of $\pi_3(S)$.
\end{proof}

Thus, by Claim~\ref{claim13}, Claim~\ref{claim14} and Claim~\ref{claim9}, we have already proved that
$$
\WS_3(2,l) = \left\{
\begin{array}{lll}
2l+2 & \text{for} & l\equiv 0,1,5\bmod{9}, \\
2l+1 & \text{for} & l\geq9,\ l\equiv 2,3,4,6,7,8\bmod{9}.
\end{array}
\right.
$$
It remains to settle the cases $l\in\{2,3,4,6,7,8\}$.

\begin{claim}\label{claim15}
$\WS_3(2,2) \leq 4$.
\end{claim}

\begin{proof}
Let $S=\{1,2,3,4,5\}$. Then $\pi_3(\{1,2,3,4,5\})=\{0,1,1,2,2\}$. In any $2$-partition of $\pi_3(S)$ there is one multiset with cardinality of at least $3$. Since the submultisets $\{0,1,1\}$, $\{0,1,2\}$, $\{0,2,2\}$, $\{1,1,2\}$ and $\{2,2,1\}$ are not weakly $2$-sum-free in $\Zt$, it follows that there is no weakly $2$-sum-free $2$-partition of $\pi_3(S)$. Therefore $\WS_3(2,2)\leq 4$.
\end{proof}

So, from Claim~\ref{claim13} and Claim~\ref{claim15}, we obtain that $\WS_3(2,2)=4$.

\begin{claim}\label{claim16}
$\WS_3(2,4)\leq 8$.
\end{claim}

\begin{proof}
Let $S=\{1,2,3,4,5,6,7,8,9\}$. Then $\pi_3(S)=\{0,0,0,1,1,1,2,2,2\}$. In any $2$-partition of $\pi_3(S)$ there is one multiset with cardinality of at least $5$. Since a weakly $4$-sum-free multiset with cardinality of $5$ contains exactly $2$ different terms of $\Zt$ by Lemma~\ref{lem3} and since the submultisets
$$
\{0,0,1,1,1\},\ \{0,0,0,1,1\},\ \{1,1,1,2,2\},\ \{1,1,2,2,2\},\ \{0,0,2,2,2\},\ \{0,0,0,2,2\}
$$
are not weakly $4$-sum-free in $\Zt$ by Lemma~\ref{lem4}, it follows that there is no weakly $4$-sum-free $2$-partition of $\pi_3(S)$. Therefore $\WS_3(2,4)\leq 8$.
\end{proof}

Thus, from Claim~\ref{claim13} and Claim~\ref{claim16}, we obtain that $\WS_3(2,4)=8$.

\begin{claim}\label{claim17}
$\WS_3(2,l)\geq 2l+1$ for $l\in\{3,6,7,8\}$.
\end{claim}

\begin{proof}
Let $S=\{1,2,\ldots,2l+1\}$. We can verify by using Lemma~\ref{lem4} that the following partition $P=\{S_1,S_2\}$ of $\pi_3(S)$ is weakly $l$-sum-free.
\begin{center}
\begin{tabular}{cc}
 & \\
$l=3$ :
&
\begin{tabular}{|c|c|c|c|c|}
\hline
$M$ & $\m_M(0)$ & $\m_M(1)$ & $\m_M(2)$ & $|M|$ \\
\hline\hline
$S_1$ & $0$ & $2$ & $2$ & $4$ \\
\hline
$S_2$ & $2$ & $1$ & $0$ & $3$ \\
\hline\hline
$\pi_3(S)$ & $2$ & $3$ & $2$ & $7$ \\
\hline
\end{tabular}\\
 & \\
$l=6$ :
& 
\begin{tabular}{|c|c|c|c|c|}
\hline
$M$ & $\m_M(0)$ & $\m_M(1)$ & $\m_M(2)$ & $|M|$ \\
\hline
$S_1$ & $3$ & $4$ & $0$ & $7$ \\
\hline
$S_2$ & $1$ & $1$ & $4$ & $6$ \\
\hline\hline
$\pi_3(S)$ & $4$ & $5$ & $4$ & $13$ \\
\hline
\end{tabular}\\
 & \\
$l=7$ :
&
\begin{tabular}{|c|c|c|c|c|}
\hline
$M$ & $\m_M(0)$ & $\m_M(1)$ & $\m_M(2)$ & $|M|$ \\
\hline
$S_1$ & $4$ & $4$ & $0$ & $8$ \\
\hline
$S_2$ & $1$ & $1$ & $5$ & $7$ \\
\hline\hline
$\pi_3(S)$ & $5$ & $5$ & $5$ & $15$ \\
\hline
\end{tabular}\\

\end{tabular}
\begin{tabular}{cc}

 & \\
$l=8$ :
&
\begin{tabular}{|c|c|c|c|c|}
\hline
$M$ & $\m_M(0)$ & $\m_M(1)$ & $\m_M(2)$ & $|M|$ \\
\hline
$S_1$ & $5$ & $4$ & $0$ & $9$ \\
\hline
$S_2$ & $0$ & $2$ & $6$ & $8$ \\
\hline\hline
$\pi_3(S)$ & $5$ & $6$ & $6$ & $17$ \\
\hline
\end{tabular}
\\
\end{tabular}
\end{center}
\end{proof}

Finally, by Claim~\ref{claim14} and Claim~\ref{claim17}, we deduce that $\WS_3(2,l)=2l+1$ for $l\in\{3,6,7,8\}$. This concludes the proof of Theorem~\ref{thm4} in this case.

\subsection{For \texorpdfstring{$k\geq3$}{k>=3} and \texorpdfstring{$l\equiv 0,2\bmod{3}$}{l=0,2 mod 3}}

Let $k\geq3$ and $l$ be two positive integers, with $l\equiv 0$ or $2 \bmod{3}$. We will prove that
$$
\WS_3(k,l) = 3(k-2)l+2.
$$

\begin{claim}
$\WS_3(k,l) \leq 3(k-2)l+2$ for $k\geq3$ and $l\equiv0,2 \bmod{3}$.
\end{claim}

\begin{proof}
Let $S=\{1,2,\ldots,3(k-2)l+3\}$. Suppose that there exists a partition $P=\{S_1,\ldots,S_k\}$ of $\pi_3(S)$ into $k$ weakly $l$-sum-free multisets in $\Zt$.

First, we prove that $|S_i|\le l+1$, for all $1\le i\le k$. From the definition of $\pi_3(S)$, we know that $\m_{\pi_3(S)}(x)=(k-2)l+1$, for all $x\in\Zt$. Since $\m_{S_i}(0)\le l$, for all $1\leq i\leq k$, by Lemma~\ref{lem3} and $\sum_{i=1}^{k}\m_{S_i}(0)=\m_{\pi_3(S)}(0)=(k-2)l+1$, it follows from the pigeonhole principle that there are at least $k-1$ multisets $S_i$ containing at least one element $0\in\Zt$. Moreover, by Lemma~\ref{lem3} again, a weakly $l$-sum-free multiset $S_i$ for which $\m_{\pi_3(S_i)}(0)\geq1$ has cardinality of at most $l+1$. Therefore, if the number of multisets $S_i$ for which $\m_{S_i}(0)\geq 1$ is $k$, we have $|S_i|\leq l+1$, for all $1\leq i\leq k$. In the other case, if the number of multisets $S_i$ for which $\m_{S_i}(0)\geq 1$ is exactly $k-1$, we can also prove that the $k$th multiset has cardinality of at most $l+1$. Indeed, if we suppose that $\m_{S_i}(0)\geq1$, for all $1\leq i\leq k-1$, and $\m_{S_k}(0)=0$ with $|S_k|\geq l+2$, then $S_k\cap\Zt=\{a\}$ with $a\in\{1,2\}\subset\Zt$ by Lemma~\ref{lem3}. Thus, since $l+2=\m_{\pi_3(S_k)}(a)\le \m_{\pi_3(S)}(a)=(k-2)l+1$, it follows that $k\geq4$. Then,
$$
2((k-2)l+1)-(k-1)(l+1) \begin{array}[t]{l}
= 2kl-4l+2-kl+l-k+1 \\
= kl-3l-k+3 \\
= (k-3)l-(k-3) \\
= (k-3)(l-1) \\
> 0. \\
\end{array}
$$
Therefore $2((k-2)l+1)>(k-1)(l+1)$. Moreover, since $S_k\cap\Zt=\{a\}$, the $2((k-2)l+1)$ elements $0$ and $2a$ of $\pi_3(S)$ must be in $S_1\cup S_2\cup\cdots\cup S_{k-1}$. But we know that $|S_i|\leq l+1$, for all $1\leq i\leq k-1$, in contradiction with $2((k-2)l+1)>(k-1)(l+1)$. We have proved that, in every cases, we have $|S_i|\leq l+1$, for all $1\leq i\leq k$. We deduce that
$$
3((k-2)l+1) = |\pi_3(S)| \leq k(l+1).
$$

Finally, since
$$
3((k-2)l+1)-k(l+1) \begin{array}[t]{l}
= 3kl-6l+3-kl-k \\
= 2kl -6l -k +3 \\
= k(2l-1) - 3(2l-1) \\
= (k-3)(2l-1), \\
\end{array}
$$
we obtain that $3((k-2)l+1)>k(l+1)$ for $k\geq4$, in contradiction with $|\pi_3(S)|\leq k(l+1)$. For $k=3$, we have $|S_1|=|S_2|=|S_3|=l+1$. Since $\m_{\pi_3(S)}(0)=l+1>l$, it follows from Lemma~\ref{lem3} that the number of multisets $S_i$ containing elements $0\in\Zt$ is at least of two. It is exactly two because if each multiset $S_i$ contains elements $0\in\Zt$, then there is a weakly $l$-sum-free multiset $S_i$ which contains all the elements of $\Zt$ and with $|S_i|=l+1$, in contradiction with Lemma~\ref{lem3}. So we deduce that the partition $P$ is either of the form
$$
S_1\cap\Zt = S_2\cap\Zt = \{0,a\}\quad \text{and}\quad S_3\cap\Zt=\{2a\},
$$
or of the form
$$
S_1\cap \Zt = \{0,a\},\quad S_2\cap \Zt = \{0,2a\}\quad \text{and}\quad S_3\cap \Zt = \{a,2a\},
$$
where $a\in\{1,2\}\subset\Zt$ and $|S_1|=|S_2|=|S_3|=l+1$. For the first form, we know from Lemma~\ref{lem4} that $S_1$ and $S_2$ are weakly $l$-sum-free if and only if $\sum_{x\in S_1}x=\sum_{x\in S_2}x=a$, that is, if and only if $\m_{S_1}(a)\equiv\m_{S_2}(a)\equiv 1\bmod{3}$. It follows that
$$
l+1 = \m_{\pi_3(S)}(a) = \m_{S_1}(a) + \m_{S_2}(a) \equiv 2 \pmod{3},
$$
and thus $l\equiv 1\bmod{3}$, in contradiction with the hypothesis that $l\equiv 0$ or $2\bmod{3}$. For the second form, we know from Lemma~\ref{lem4} that $S_1$ and $S_2$ are weakly $l$-sum-free if and only if $\sum_{x\in S_1}x=a$ and $\sum_{x\in S_2}x=2a$, that is, if and only if $\m_{S_1}(a)\equiv\m_{S_2}(a)\equiv 1\bmod{3}$. Moreover, we have $\m_{S_2}(2a)=l+1-\m_{S_3}(2a)=\m_{S_3}(a)$. It follows that
$$
l+1 = \m_{\pi_3(S)}(a) = \m_{S_1}(a) + \m_{S_3}(a) = \m_{S_1}(a) + \m_{S_2}(2a) \equiv 2 \pmod{3},
$$
and thus $l\equiv 1\bmod{3}$, in contradiction with the hypothesis that $l\equiv 0$ or $2\bmod{3}$. This concludes the proof.
\end{proof}

\begin{claim}
$\WS_3(k,l) \geq 3(k-2)l+2$ for $k\geq3$ and $l\equiv0,2 \bmod{3}$.
\end{claim}

\begin{proof}
Let $S=\{1,2,\ldots,3(k-2)l+2\}$. We consider the following partition $P=\{S_1,\ldots,S_k\}$ of $\pi_3(S)$ and we prove that it is weakly $l$-sum-free.
\begin{center}
\begin{tabular}{|c|c|c|c|c|}
\hline
$M$ & $\m_M(0)$ & $\m_M(1)$ & $\m_M(2)$ & $|M|$ \\
\hline\hline
$S_1$ & $0$ & $(k-2)l+1$ & $0$ & $(k-2)l+1$ \\
\hline
$S_2$ & $0$ & $0$ & $(k-2)l+1$ & $(k-2)l+1$ \\
\hline
$S_3,\ldots,S_k$ & $l$ & $0$ & $0$ & $l$ \\
\hline\hline
$\pi_3(S)$ & $(k-2)l$ & $(k-2)l+1$ & $(k-2)l+1$ & $3(k-2)l+2$ \\
\hline
\end{tabular}
\end{center}
First, $P$ is a partition of $\pi_3(S)$ since the multiplicity functions verify that $\sum_{i=1}^{k}\m_{S_i}(x)=\m_{\pi_3(S)}(x)$ for all $x\in\Zt$. The multisets $S_1$ and $S_2$ are weakly $l$-sum-free because, as already remarked above, a multiset which is only constituted by elements $1$ or $2\in\Zt$ is always weakly $l$-sum-free when $l\equiv 0,2\bmod{3}$. For $S_3,\ldots,S_k$, multisets containing only $l$ elements are always weakly $l$-sum-free in $\Zt$. This completes the proof.
\end{proof}

\subsection{For \texorpdfstring{$k\geq 3$}{k>=3} and \texorpdfstring{$l\equiv 1\bmod{3}$}{l=1 mod 3}}

Let $k\geq 3$ and $l\ge2$ be two positive integers, with $l\equiv 1\bmod{3}$. We will prove that
$$
\WS_3(k,l) = \left\{
\begin{array}{lll}
k(l+1) & \text{for} & \left\{
\begin{array}{l}
k=3,\ k\geq5\ \text{and}\ l\geq2,\ l\equiv1\bmod{3},\\
k=4\ \text{and}\ l\geq2,\ l\equiv1,7\bmod{9},
\end{array}
\right.\\
4l+3 & \text{for} & k=4\ \text{and}\ l\equiv 4\bmod{9}.
\end{array}
\right.
$$

\begin{claim}\label{claim5}
$\WS_3(k,l)\leq k(l+1)$ for $k\geq1$ and $l\ge2,\ l\equiv1\bmod{3}$.
\end{claim}

\begin{proof}
Directly follows from Lemma~\ref{lem3}.
\end{proof}

\begin{claim}\label{claim8}
Let $l\equiv1\bmod{3}$. Let $M$ be a multiset of $\Zt$ such that $|M\cap\Zt|=2$ and $|M|=l+1$. Then, $M$ is weakly $l$-sum-free if and only if $\m_M(x)\equiv 1\bmod{3}$ for all $x\in M\cap\Zt$.
\end{claim}

\begin{proof}
Suppose that $M\cap\Zt = \{a,b\}$ with $a\neq b$. First, since $l\equiv 1\bmod{3}$, it follows that
$$
\m_{M}(a) + \m_{M}(b) = l+1 \equiv 2 \pmod{3}.
$$
Thus either $\m_{M}(a)\equiv \m_{M}(b)\equiv 1\bmod{3}$ or $\m_{M}(a)\equiv 0\bmod{3}$ and $\m_{M}(b)\equiv 2\bmod{3}$. Suppose that $\m_{M}(a)\equiv 0\bmod{3}$ and $\m_{M}(b)\equiv 2\bmod{3}$. We know, by Lemma~\ref{lem4}, that
$$
\sum_{x\in M}x = \m_{M}(a).a + \m_{M}(b).b = 2.2(a+b) = a+b.
$$
This leads to the equality
$$
2b = 0.a + 2.b = \m_{M}(a).a + \m_{M}(b).b = a+b,
$$
in contradiction with $a\neq b$. This completes the proof.
\end{proof}

\begin{claim}\label{claim4}
$\WS_3(k,l)\geq \WS_3(k-3,l)+3(l+1)$ for $k\geq4$ and $l\equiv 1\bmod{3}$.
\end{claim}

\begin{proof}
Let $S=\{1,2,\ldots,\WS_3(k-3,l)\}$. Let $P=\{S_1,\ldots,S_{k-3}\}$ be a partition of $\pi_3(S)$ into weakly $l$-sum-free multisets of $\Zt$. Consider the multisets $S_{k-2}$, $S_{k-1}$ and $S_k$ of $\Zt$ defined below by their multiplicity function.
\begin{center}
\begin{tabular}{|c|c|c|c|c|}
\hline
$M$ & $\m_M(0)$ & $\m_M(1)$ & $\m_M(2)$ & $|M|$ \\
\hline\hline
$S_{k-2}$ & $l$ & $1$ & $0$ & $l+1$ \\
\hline
$S_{k-1}$ & $0$ & $l$ & $1$ & $l+1$ \\
\hline
$S_{k}$ & $1$ & $0$ & $l$ & $l+1$ \\
\hline\hline
$S_{k-2}\cup S_{k-1}\cup S_{k}$ & $l+1$ & $l+1$ & $l+1$ & $3(l+1)$ \\
\hline
\end{tabular}
\end{center}
Then $P'=\{S_1,\ldots,S_k\}$ is a partition of the multiset $\pi_3(S')=\pi_3(\{1,2,\ldots,\WS_3(k-3,l)+3(l+1)\})$ since $\sum_{i=1}^{k}\m_{S_i}(x)=\m_{\pi_3(S)}(x)+(l+1)=\m_{\pi_3(S')}(x)$ for all $x\in\Zt$. Moreover, since $l\equiv 1\bmod{3}$, it follows that the multisets $S_{k-2}$, $S_{k-1}$ and $S_k$ are weakly $l$-sum-free in $\Zt$ by Lemma~\ref{lem4} and Claim~\ref{claim8}. Therefore $\WS_3(k,l)\geq \WS_3(k-3,l)+3(l+1)$ for all $k\geq4$ and $l\equiv 1\bmod{3}$.
\end{proof}

\begin{claim}\label{claim7}
$\WS_3(3,l) \geq 3(l+1)$ for $l\equiv1\bmod{3}$.
\end{claim}

\begin{proof}
Let $S=\{1,2,\ldots,3(l+1)\}$. The following $3$-partition $P=\{S_1,S_2,S_3\}$ of $\pi_3(S)$ is weakly $l$-sum-free, as already seen in the proof of Claim~\ref{claim4}.
\begin{center}
\begin{tabular}{|c|c|c|c|c|}
\hline
$M$ & $\m_M(0)$ & $\m_M(1)$ & $\m_M(2)$ & $|M|$ \\
\hline\hline
$S_{1}$ & $l$ & $1$ & $0$ & $l+1$ \\
\hline
$S_{2}$ & $0$ & $l$ & $1$ & $l+1$ \\
\hline
$S_{3}$ & $1$ & $0$ & $l$ & $l+1$ \\
\hline\hline
$\pi_3(S)$ & $l+1$ & $l+1$ & $l+1$ & $3(l+1)$ \\
\hline
\end{tabular}
\end{center}
This completes the proof.
\end{proof}

Let $k$ be a positive integer and let $S=\{1,2,\ldots,k(l+1)\}$. In the sequel of this section, suppose that there exists a partition $P=\{S_1,\ldots,S_k\}$ of $\pi_3(S)$ into $k$ weakly $l$-sum-free multisets of $\Zt$. By Lemma~\ref{lem3}, $|S_i|=l+1$ and $|S_i\cap\Zt|=2$ for all $1\leq i\leq k$. For every $x\in\Zt$, denote by $n_x$ the number of multisets $S_i$ in $P$ such that $x\in S_i\cap\Zt$. For $S_i\cap\Zt=\{a,b\}$, the multiset $S_i$ is weakly $l$-sum-free if and only if
$$
\m_{S_i}(a)\equiv \m_{S_i}(b) \equiv1\pmod{3}
$$
by Lemma~\ref{lem4} and Claim~\ref{claim8}. It follows that
$$
n_x\in\{1,\ldots,k\}\quad \text{and}\quad n_x \equiv \m_{\pi_3(S)}(x) \pmod{3},
$$
for all $x\in\Zt$. We distinguish different cases depending on the value of $k$ and the residue class of $l$ modulo $9$.

\begin{claim}\label{claim18}
$\WS_3(5,l) \geq 5(l+1)$ for $l\equiv4\bmod{9}$.
\end{claim}

\begin{proof}
Consider the euclidean division of $l$ by $9$, that is $l=9r+4$. Then,
$$
5(l+1)=5(9r+5)=45r+25=3(15r+8)+1
$$
and
$$
\begin{array}{l}
n_0 \equiv \m_{\pi_3(S)}(0) = 15r + 8 \equiv 2 \pmod{3},\\[2ex]
n_1 \equiv \m_{\pi_3(S)}(1) = 15r+9 \equiv0 \pmod{3},\\[2ex]
n_2 \equiv \m_{\pi_3(S)}(2) = 15r+8 \equiv2 \pmod{3}.
\end{array}
$$
We deduce that $n_0,n_2\in\{2,5\}$ and $n_1=3$. Moreover, we have $n_0+n_1+n_2=10$. For $n_0=5$, $n_1=3$ and $n_2=2$, we can verify by using Lemma~\ref{lem4} and Claim~\ref{claim8} that the following partition of $\pi_3(S)$ is weakly $l$-sum-free.
\begin{center}
\begin{tabular}{|c||c|c|c||c|}
\hline
$M$ & $\m_M(0)$ & $\m_M(1)$ & $\m_M(2)$ & $|M|$ \\
\hline\hline
$S_1$ & $9r+4$ & $1$ & $0$ & $l+1$ \\
\hline
$S_2$ & $3r+1$ & $6r+4$ & $0$ & $l+1$ \\
\hline
$S_3$ & $1$ & $9r+4$ & $0$ & $l+1$ \\
\hline
$S_4$ & $3r+1$ & $0$ & $6r+4$ & $l+1$ \\
\hline
$S_5$ & $1$ & $0$ & $9r+4$ & $l+1$ \\
\hline\hline
$\pi_3(S)$ & $15r+8$ & $15r+9$ & $15r+8$ & $5(l+1)$ \\
\hline
\end{tabular}
\end{center}
This completes the proof.
\end{proof}

\begin{claim}\label{claim19}
$\WS_3(5,l) \geq 5(l+1)$ for $l\equiv7\bmod{9}$.
\end{claim}

\begin{proof}
Consider the euclidean division of $l$ by $9$, that is $l=9r+7$. Then,
$$
5(l+1)=5(9r+8)=45r+40=3(15r+13)+1
$$
and
$$
\begin{array}{l}
n_0 \equiv \m_{\pi_3(S)}(0) = 15r + 13 \equiv 1 \pmod{3},\\[2ex]
n_1 \equiv \m_{\pi_3(S)}(1) = 15r+14 \equiv 2 \pmod{3},\\[2ex]
n_2 \equiv \m_{\pi_3(S)}(2) = 15r+13 \equiv 1 \pmod{3}.
\end{array}
$$
We deduce that $n_0,n_2\in\{1,4\}$ and $n_1\in\{2,5\}$. Moreover, we have $n_0+n_1+n_2=10$. For $n_0=n_2=4$ and $n_1=2$, we can verify by using Lemma~\ref{lem4} and Claim~\ref{claim8} that the following partition of $\pi_3(S)$ is weakly $l$-sum-free.
\begin{center}
\begin{tabular}{|c||c|c|c||c|}
\hline
$M$ & $\m_M(0)$ & $\m_M(1)$ & $\m_M(2)$ & $|M|$ \\
\hline\hline
$S_1$ & $3r+1$ & $6r+7$ & $0$ & $l+1$ \\
\hline
$S_2$ & $0$ & $9r+7$ & $1$ & $l+1$ \\
\hline
$S_3$ & $9r+7$ & $0$ & $1$ & $l+1$ \\
\hline
$S_4$ & $3r+4$ & $0$ & $6r+4$ & $l+1$ \\
\hline
$S_5$ & $1$ & $0$ & $9r+7$ & $l+1$ \\
\hline\hline
$\pi_3(S)$ & $15r+13$ & $15r+14$ & $15r+13$ & $5(l+1)$ \\
\hline
\end{tabular}
\end{center}
This completes the proof.
\end{proof}

\begin{claim}\label{claim20}
$\WS_3(4,l) \geq 4(l+1)$ for $l\equiv1\bmod{9}$.
\end{claim}

\begin{proof}
Consider the euclidean division of $l$ by $9$, that is $l=9r+1$. Then,
$$
4(l+1)=4(9r+2)=36r+8=3(12r+2)+2
$$
and
$$
\begin{array}{l}
n_0 \equiv \m_{\pi_3(S)}(0) = 12r + 2 \equiv 2 \pmod{3},\\[2ex]
n_1 \equiv \m_{\pi_3(S)}(1) = 12r + 3 \equiv 0 \pmod{3},\\[2ex]
n_2 \equiv \m_{\pi_3(S)}(2) = 12r + 3 \equiv 0 \pmod{3}.
\end{array}
$$
We deduce that $n_0=2$ and $n_1=n_2=3$. Moreover, we have $n_0+n_1+n_2=8$. We can verify by using Lemma~\ref{lem4} and Claim~\ref{claim8} that the following partition of $\pi_3(S)$ is weakly $l$-sum-free.
\begin{center}
\begin{tabular}{|c||c|c|c||c|}
\hline
$M$ & $\m_M(0)$ & $\m_M(1)$ & $\m_M(2)$ & $|M|$ \\
\hline\hline
$S_1$ & $6r+1$ & $3r+1$ & $0$ & $l+1$ \\
\hline
$S_2$ & $6r+1$ & $0$ & $3r+1$ & $l+1$ \\
\hline
$S_3$ & $0$ & $6r+1$ & $3r+1$ & $l+1$ \\
\hline
$S_4$ & $0$ & $3r+1$ & $6r+1$ & $l+1$ \\
\hline\hline
$\pi_3(S)$ & $12r+2$ & $12r+3$ & $12r+3$ & $4(l+1)$ \\
\hline
\end{tabular}
\end{center}
This completes the proof.
\end{proof}

\begin{claim}
$\WS_3(4,l) = 4l+3$ for $l\equiv4\bmod{9}$.
\end{claim}

\begin{proof}
Consider the euclidean division of $l$ by $9$, that is $l=9r+4$. Then,
$$
4(l+1)=4(9r+5)=36r+20=3(12r+6)+2
$$
and
$$
\begin{array}{l}
n_0 \equiv \m_{\pi_3(S)}(0) = 12r + 6 \equiv 0 \pmod{3},\\[2ex]
n_1 \equiv \m_{\pi_3(S)}(1) = 12r + 7 \equiv 1 \pmod{3},\\[2ex]
n_2 \equiv \m_{\pi_3(S)}(2) = 12r + 7 \equiv 1 \pmod{3}.
\end{array}
$$
We deduce that $n_0=3$ and $n_1,n_2\in\{1,4\}$. Moreover, $n_1,n_2\neq1$ since $\m_{\pi_3(S)}(1)=\m_{\pi_3(S)}(2)=12r+7>l+1$. Thus $n_0=3$ and $n_1=n_2=4$ in contradiction with $n_0+n_1+n_2=8$. This proves that
$$
\WS_3(4,l) \leq 4l+3\quad \text{for}\ l\equiv4\pmod{9}.
$$
Finally, from Claim~\ref{claim4} and Claim~\ref{claim10}, we deduce that $\WS_3(4,l)\geq \WS_3(1,l) + 3(l+1) = 4l+3$. This completes the proof.
\end{proof}

\begin{claim}\label{claim21}
$\WS_3(7,l) \geq 7(l+1)$ for $l\equiv4\bmod{9}$.
\end{claim}

\begin{proof}
Consider the euclidean division of $l$ by $9$, that is $l=9r+4$. Then,
$$
7(l+1)=7(9r+5)=63r+35=3(21r+11)+2
$$
and
$$
\begin{array}{l}
n_0 \equiv \m_{\pi_3(S)}(0) = 21r + 11 \equiv 2 \pmod{3},\\[2ex]
n_1 \equiv \m_{\pi_3(S)}(1) = 21r + 12 \equiv 0 \pmod{3},\\[2ex]
n_2 \equiv \m_{\pi_3(S)}(2) = 21r + 12 \equiv 0 \pmod{3}.
\end{array}
$$
We deduce that $n_0\in\{2,5\}$ and $n_1,n_2\in\{3,6\}$. Moreover, we have $n_0+n_1+n_2=14$. Since $\m_{\pi_3(S)}(0)=21r+11>18r+8=2l$, it follows that $n_0=5$. For $n_0=5$, $n_1=6$ and $n_2=3$, we can verify by using Lemma~\ref{lem4} and Claim~\ref{claim8} that the following partition of $\pi_3(S)$ is weakly $l$-sum-free.
\begin{center}
\begin{tabular}{|c||c|c|c||c|}
\hline
$M$ & $\m_M(0)$ & $\m_M(1)$ & $\m_M(2)$ & $|M|$ \\
\hline\hline
$S_1$ & $9r+4$ & $1$ & $0$ & $l+1$ \\
\hline
$S_2$ & $9r+4$ & $1$ & $0$ & $l+1$ \\
\hline
$S_3$ & $3r+1$ & $6r+4$ & $0$ & $l+1$ \\
\hline
$S_4$ & $1$ & $9r+4$ & $0$ & $l+1$ \\
\hline
$S_5$ & $1$ & $0$ & $9r+4$ & $l+1$ \\
\hline
$S_6$ & $0$ & $6r+1$ & $3r+4$ & $l+1$ \\
\hline
$S_7$ & $0$ & $1$ & $9r+4$ & $l+1$ \\
\hline\hline
$\pi_3(S)$ & $21r+11$ & $21r+12$ & $21r+12$ & $7(l+1)$ \\
\hline
\end{tabular}
\end{center}
This completes the proof.
\end{proof}

\begin{claim}\label{claim22}
$\WS_3(4,l) \geq 4(l+1)$ for $l\equiv7 \bmod{9}$.
\end{claim}

\begin{proof}
Consider the euclidean division of $l$ by $9$, that is $l=9r+7$. Then,
$$
4(l+1)=4(9r+8)=36r+32=3(12r+10)+2
$$
and
$$
\begin{array}{l}
n_0 \equiv \m_{\pi_3(S)}(0) = 12r + 10 \equiv 1 \pmod{3},\\[2ex]
n_1 \equiv \m_{\pi_3(S)}(1) = 12r + 11 \equiv 2 \pmod{3},\\[2ex]
n_2 \equiv \m_{\pi_3(S)}(2) = 12r + 11 \equiv 2 \pmod{3}.
\end{array}
$$
We deduce that $n_0\in\{1,4\}$ and $n_1=n_2=2$. Since $\m_{\pi_3(S)}(0)=12r+10>9r+7=l$, it follows that $n_0=4$. For $n_0=4$ and $n_1=n_2=2$, we can verify by using Lemma~\ref{lem4} and Claim~\ref{claim8} that the following partition of $\pi_3(S)$ is weakly $l$-sum-free.
\begin{center}
\begin{tabular}{|c||c|c|c||c|}
\hline
$M$ & $\m_M(0)$ & $\m_M(1)$ & $\m_M(2)$ & $|M|$ \\
\hline\hline
$S_1$ & $3r+4$ & $6r+4$ & $0$ & $l+1$ \\
\hline
$S_2$ & $3r+1$ & $6r+7$ & $0$ & $l+1$ \\
\hline
$S_3$ & $3r+4$ & $0$ & $6r+4$ & $l+1$ \\
\hline
$S_4$ & $3r+1$ & $0$ & $6r+7$ & $l+1$ \\
\hline\hline
$\pi_3(S)$ & $12r+10$ & $12r+11$ & $12r+11$ & $4(l+1)$ \\
\hline
\end{tabular}
\end{center}
This completes the proof.
\end{proof}

Finally, we obtain that $\WS_3(k,l) = k(l+1)$ for all $k\geq 3$ and $l\geq2$, with $l\equiv 1\bmod{3}$, except for $k=4$ and $l\equiv 4\bmod{9}$, by combining Claim~\ref{claim5}, Claim~\ref{claim4}, Claim~\ref{claim7}, Claim~\ref{claim18}, Claim~\ref{claim19}, Claim~\ref{claim20}, Claim~\ref{claim21} and Claim~\ref{claim22}.\\[2ex]
This concludes the proof of Theorem~\ref{thm4}.

\section{Conclusion and future work}

In this paper, the modular generalized Schur numbers $\Sc_m(k,l)$ and the modular
generalized weak Schur numbers $\WS_m(k,l)$ have been explicitly determined for all positive integers $k$ and $l$ and for small moduli $m\in\{1,2,3\}$. Although the determination of the exact values of $\Sc_m(k,l)$ and of $\WS_m(k,l)$ seems to be much more difficult for moduli $m \geq 4$, we can hope to find nontrivial lower and upper bounds of these numbers.

\subsection*{Acknowledgements}
We would like to thank Shalom Eliahou for his constructive comments and suggestions throughout the preparation of this paper.



\end{document}